\documentclass[reqno]{amsart}
\usepackage[T1]{fontenc}
\usepackage{amssymb,bbm}
\usepackage{amsmath}
\usepackage{amsfonts}
\usepackage{ifthen}
\usepackage{leftidx}
\usepackage{comment}
\usepackage{mathtools}

\newcounter{margnotes}

\def\sideremark#1{\ifvmode\leavevmode\fi\vadjust{\vbox to0pt{\vss 
      \hbox to 0pt{\hskip\hsize\hskip1em           
 \vbox{\hsize3cm\tiny\raggedright\pretolerance10000
 \noindent #1\hfill}\hss}\vbox to8pt{\vfil}\vss}}}%
                        
\newcommand{\edz}[1]{}

\newcounter{lemenumi}
\newcommand{\labelemenumi}{(\alph{lemenumi})}

\newtheorem{theorem}{Theorem}[section]
\newtheorem{lemma}[theorem]{Lemma}
\newtheorem{conjecture}[theorem]{Conjecture}
\newtheorem{proposition}[theorem]{Proposition}
\newtheorem{corollary}[theorem]{Corollary}

\theoremstyle{definition}
\newtheorem{definition}[theorem]{Definition}
\newtheorem{example}[theorem]{Example}

\theoremstyle{remark}
\newtheorem{remark}[theorem]{Remark}

\numberwithin{equation}{section}

\newcommand{\Span}{\mathrm{Span}}

\newcommand{\OO}{\mathcal{O}}

\begin{document}
\newcommand{\xc}{\xi}
\newcommand{\yc}{\eta}
\newcommand{\ep}{\varepsilon}
\newcommand{\jb}{II}
\newcommand{\df}{d}
\newcommand{\al}{\alpha}

\newcommand{\testrat}{R}
\newcommand{\testpol}{P}

\newcommand{\pn}{Q}
\newcommand{\qn}{Q_1}
\newcommand{\mon}{\mathcal{M}on}
\newcommand{\ve}{\varepsilon}
\newcommand{\var}{\mathcal{V}ar}
\newcommand{\org}{T}
\newcommand{\orgg}{t}
\newcommand{\supp}{\text{supp}}
\newcommand{\R}{\mathbb{R}}
\newcommand{\C}{\mathbb{C}}
\newcommand{\N}{\mathbb{N}}
\newcommand{\Z}{\mathbb{Z}}
\newcommand{\Q}{\mathbb{Q}}
\newcommand{\Xbar}{\bold{X}}


\newcommand{\edg}{\gamma^0}
\newcommand{\vtx}{p}
\newcommand{\CH}{{\mathbb{C}} H_1({\mathcal{O}})}
\newcommand{\CHy}{{\mathbb{C}} H_1({\mathcal{O}_y})}

\title[The first nonzero Melnikov function]{Bounding the length of iterated integrals of the first nonzero Melnikov function} 
\author[P.  Marde\v si\'c]{Pavao Marde\v si\'c}
\address[P. Marde\v si\'c]{Universit\'e de Bourgogne, Institute de Math\'ematiques de Bourgogne - UMR 5584 CNRS\\
	Universit\'e de Bourgogne,
	9 avenue Alain Savary,
	BP 47870, 21078 Dijon, FRANCE}
\email{mardesic@u-bourgogne.fr}
\thanks{This work was supported by UNAM PREI Dgapa, Franco-Mexican project LAISLA, Conacyt 219722, Papiit, Dgapa UNAM IN106217, ECOS Nord-Conacyt 249542 and Conacyt 291121}

\author[D. Novikov]{Dmitry Novikov}
\address[D. Novikov and J. Pontigo-Herrera]{Faculty of Mathematics and Computer Science, Weizmann Institute of Science, Rehovot, Israel}
\email{dmitry.novikov@weizmann.ac.il}

\author[L. Ortiz-Bobadilla]{Laura Ortiz-Bobadilla}
\address[P.  Marde\v si\'c, L. Ortiz-Bobadilla and J. Pontigo-Herrera]{Instituto de Matem\'aticas, Universidad Nacional Aut\'onoma de M\'exico (UNAM)\\
	\'Area de la Investigaci\'on Cient\'ifica, Circuito exterior, Ciudad Universitaria, 04510, Ciudad de M\'exico, M\'exico}
\email{laura@matem.unam.mx}

\author[J. Pontigo-Herrera]{Jessie Pontigo-Herrera}
\email{pontigo@matem.unam.mx}

\begin{abstract}
	We consider small polynomial deformations of integrable systems of the form $dF=0$, $F\in\mathbb C[x,y]$ and the first nonzero term $M_\mu$ of the displacement function $\Delta(t,\epsilon)=\sum_{i=\mu}M_i(t)\epsilon^i$ along a cycle $\gamma(t)\in F^{-1}(t)$. It is known that $M_\mu$ is an iterated integral of length at most $\mu$.	The bound $\mu$ depends on the deformation of $dF$.
	
	In this paper we give a \emph{ universal bound} for the length of the iterated integral expressing the first nonzero term $M_\mu$ depending only on the geometry of the unperturbed system $dF=0$. The result generalizes the result of Gavrilov and Iliev providing a sufficient condition for $M_\mu$ to be  given by an abelian integral i.e. by an iterated integral of length $1$.
	We conjecture that our bound is optimal. 
\end{abstract}
\maketitle
\section{Introduction and main results}
In this paper we study small one-parameter polynomial deformations of planar polynomial Hamiltonian systems. They can be written in the form
\begin{equation}\label{deformation}
	dF+\sum_{i=1}^\infty\epsilon^{i}\eta_i=0,
\end{equation} 
where $F\in\C[x,y]$ is a polynomial and  $\eta_i$ are polynomial one-forms. 

Let   $\gamma(t)$ be a continuous family of loops in the first homotopy group $\pi_1(F^{-1}(t),p(t))$ of leaves $\{F=t\}$ of the foliation $F:\C^2\to\C$, with a base point $p(t)$ chosen on an analytic  transversal $\tau$ to a generic leaf $\{F=t_0\}$. Consider the \emph{displacement function}  $\Delta$ of \eqref{deformation} along $\gamma$, that is the first return map (holonomy) along $\gamma$  minus identity, on the transversal $\tau$ parametrized by the values $t$ of the Hamiltonian $F$,
$$
\Delta(t,\epsilon)=\sum_{i=\mu}^\infty M_i(t)\epsilon^i,
$$
with $M_\mu\not\equiv0$. The functions $M_i$ are called \emph{Melnikov functions along $\gamma(t)$} and $M_\mu$ is the \emph{first nonzero Melnikov function along $\gamma(t)$}.
These functions are sometimes called (principal)  Poincar\'e-Pontryagin functions, as in  \cite{G,PU,U1,U2}.  

For a fixed $\epsilon$, the isolated zeros of $\Delta(\epsilon, \cdot)$ correspond to the  \emph{limit cycles} of \eqref{deformation}. For $\epsilon=0$, there are no isolated zeros, and the Hamiltonian system $dF=0$ has no limit cycles. 
The real counterpart of these problem, i.e. real perturbations of real planar polynomial Hamiltonian vector field, appear naturally in investigations related to Hilbert 16th problem. In particular, the famous \emph{Arnol'd-Hilbert's problem}, see  \cite{A}, asks  for a bound for the number of limit cycles born for small $\epsilon$.

It is known that near the   singular points of  $\Delta$ there is no one-to-one correspondence between the isolated zeros of $\Delta$ and $M_\mu$ (\emph{alien cycles} \cite{R}). Still, the first nonzero Melnikov function $M_\mu$ carries essential information on the displacement function $\Delta.$

If $\mu=1$, then $M_\mu$ is an \emph{abelian integral} and its zeros have been thoroughly studied, see \cite{BNY1}.  
If $M_1\equiv0$, one searches for the first nonzero Melnikov function $M_\mu$. 
The special case of  \eqref{deformation} of linear deformations
\begin{equation}\label{eq:linear deformation}
dF+\epsilon\eta=0
\end{equation}
was studied by many authors.
Fran\c coise \cite{F} (see also \cite{Y}) gave an algorithm for calculating $M_\mu$ of \eqref{eq:linear deformation} under a certain condition (called condition (*) of Fran\c coise, cf. Remark \ref{cond*}). The algorithm assures that under that assumption, $M_\mu$ is always an abelian integral. 
In \cite{JMP1} and \cite{JMP2} the question of non-abelian character of $M_\mu$ was raised if Fran\c coise condition was not verified. 
Examples where $M_\mu$ is not an abelian integral appear in \cite{Z}, \cite{Il} and \cite{U1}.
Gavrilov \cite{G} showed that in general $M_\mu$ is an \emph{iterated integral} of \emph{length} $\ell\leq\mu$. 
Iterated integrals (see Section~\ref{sec:Iterated Integrals} for precise definitions) are generalizations of abelian integrals. Abelian integrals are iterated integrals of length $1$. As length grows, iterated integrals become more complicated and discern finer properties of foliations defined in terms of the fundamental groups of leaves rather than just their homological properties.  Considered as functions of $t$, iterated integrals belong to a wider class of Q-functions defined in \cite{BNY2}, see also \cite{BN}, with complexity explicitly bounded in terms of their length and degrees of $F$ and of the forms $\eta_1$.

It happens frequently \cite{U1,U2,P} that $M_\mu$  is in fact of length smaller than $\mu$. Of course, iterated integrals of lower length are simpler. Hence, characterizing  the length of $M_\mu$ as an iterated integral is a natural problem. 
More important, the number $\mu$ of the first nonzero Melnikov function depends essentially on the perturbation, i.e. on the forms $\eta_i$. On the contrary, the {length} of $M_\mu$ as an iterated integral often  depends only  on the topology of the foliation defined by $F$  and the cycle $\gamma$. This is known  in the generic case, see Corollary~\ref{cor:IlyaFranc}, and also in some non-generic cases, see \cite{P, U1,U2}.

In \cite{GI}, Gavrilov and Iliev gave a sufficient condition  for the first nonzero Melnikov function to be an abelian integral, i.e. an iterated integral of length $k=1$. 

In this paper, we show that this condition is necessary and  generalize their result  by giving a sufficient condition for the length of the first nonzero Melnikov function $M_\mu$ of any deformation \eqref{deformation} of $dF$ to be  at most $k\geq1$. 
We define the normal subgroup $\OO$ of  the fundamental group $\pi_1(F^{-1}(t_0), p(t_0))$ of a non-singular curve $F^{-1}(t_0)$ generated by the orbit of $\gamma(t)$ under the action of the monodromy of the foliation of $\C^2$ defined by $F$.
The upper bound on the length of the iterated integral representing $M_\mu$ is given in terms of the position of $\OO$ with respect to the filtration of $\pi_1(F^{-1}(t_0), p(t_0))$ by its lower central series.

\begin{remark}
		Let us stress that an important consequence of our Theorem \ref{main}  is that the bound $k$ on the length of the  iterated integrals representing the Melnikov function $M_\mu$ is universal, i.e. does not depend on the perturbation, whereas $\mu$ certainly depends on $\eta_i$. Our bound depends only on the unperturbed system $dF=0$ and the cycle $\gamma$.
\end{remark}
\edz{By curve selection lemma, one can generalize to multi-parameter perturbation. }

\subsection{First homology of the orbit of a loop $\gamma$}
The first nonzero Melnikov function  depends only on the class of $\gamma$ in the free homotopy group of the fiber, see Section~\ref{sec:Melnikov}. 
In this section we define a group characterizing the underlying topological properties of the fibration defined by $F$ relative to $\gamma$. This group is the key ingredient of the main theorem. 

Consider the   fibration $\mathcal{F}$ defined by $F:\C^2\setminus F^{-1}(\Sigma)\to \C\setminus\Sigma$, where $\Sigma$ is a finite set of atypical values of $F$. Choose a typical value $t_0\in\C\setminus\Sigma$, $p_0=p(t_0)$,  and let  $\gamma(t_0)\in\pi_1\left(F^{-1}(t_0),p_0\right)$.

Choose any path  $\alpha$ in $\C^2\setminus F^{-1}(\Sigma)$  starting at $p_0$ and ending at $p_1$. Using Ehresmann connection, we transport $\gamma$ along ${\alpha}$. This gives an isomorphism 
\begin{equation}\label{eq:i_alpha}
i_{\alpha}: \pi_1\left(F^{-1}(t_0),p_0\right)\to \pi_1\left(F^{-1}(F(p_1)),p_1\right)
\end{equation}
and, in particular, an isomorphism of the fundamental groups of connected components of $F^{-1}(t_0)$ if there are several components.

If $\alpha$ is closed, $i_\alpha$ is 	a well-defined  automorphism of $\pi_1\left(F^{-1}(t_0),p_0\right)$, which depends only on the homotopy class of $\alpha\in \pi_1\left(\C^2\setminus F^{-1}(\Sigma), p_0\right)$. We call it \emph{the monodromy along $\alpha$} and denote it $Mon_\alpha$.

 Here, and in the sequel, we put 
 $$
 \pi_1=\pi_1\left(F^{-1}(t_0),p_0\right).
 $$
 
We define \emph{the orbit of $\gamma$ under monodromy}
$$
\OO_{p_0}(\gamma)\lhd\pi_1
$$
as the smallest normal subgroup of $\pi_1$ containing all loops $Mon_{{\alpha}}(\gamma)$ for all $\alpha\in \pi_1\left(\C^2\setminus F^{-1}(\Sigma), p_0\right)$.

Let $K_{p_0}(\gamma)\vartriangleleft\mathcal{O}$ be the normal subgroup $K_{p_0}(\gamma)=[\mathcal{O}_{p_0}(\gamma),\pi_1]$ of $\mathcal{O}_{p_0}(\gamma)$. 

Consider the lower central sequence of $\pi_1$
$$
L_1=\pi_1\supset L_2\supset..., \quad L_{i+1}=[L_i,\pi_1].$$ 
\begin{remark}
	As the typical level sets $F^{-1}(t)$ are non-compact smooth connected Riemann surfaces, their fundamental groups $\pi_1\left(F^{-1}(t_0),p_0\right)$ are finitely generated  free groups. This will be used  without explicit reference.
\end{remark}
\subsection{Main results}

We will express the criterion for bounding the length of iterated integrals of the first nonzero Melnikov function in terms of a comparison of $K$ and the groups $\OO\cap L_{j}$, where we denote $\OO=\OO_{p_0}(\gamma)$ and $K=K_{p_0}(\gamma)$.  
\begin{definition}
	The \textit{orbit depth} $\kappa$ of $\gamma$ is defined as 
	\begin{equation}\label{kappa}
	\kappa=\min\,\Big(\{j\ge 1\,\vert\,\OO\cap L_{j+1}\subset K\}\cup\left\{+\infty\right\}\Big).
	\end{equation}
\end{definition}
Lemma~\ref{lem:ker phi_i} implies that $\OO\cap L_{\kappa'}\subset K$ for any $\kappa'>\kappa$.

\begin{theorem}\label{thm:main minus}
	The length of the first nonzero Melnikov function $M_\mu$ of any deformation \eqref{deformation} does not exceed the orbit depth of $\gamma$. 
\end{theorem}

\begin{remark}
	Here is a motivation of the result. If $\kappa=+\infty$, then there is nothing to prove. Otherwise,
	by definition of $\kappa$, any element of $\OO$ lying in $L_{\kappa+1}$ is a product of commutators of monodromy images of $\gamma(t)$ with some elements of $\pi_1$.
	So, by properties of Melnikov function, $M_\mu$ necessarily vanishes on these elements, see Section~\ref{sec:Melnikov}. But iterated integrals of length $\le \kappa$ also vanish on $L_{\kappa+1}$. Moreover, they  distinguish elements of $\pi_1/L_{\kappa+1}$: if all iterated integrals of length $\le \kappa$ vanish on  some $\sigma\in\pi_1$, then $\sigma\in L_{\kappa+1}$. 	
\end{remark}

Note that the first nonzero Melnikov function $M_\mu$ is additive on $\OO$, so necessarily vanishes on  elements whose multiples are in $K$. This motivates the following 
\begin{definition}
	The \textit{torsion-free orbit depth} $k$ of $\gamma$ is defined as
\begin{equation}\label{k}
k=\min\,\Bigg(\left\{j\ge 1\,\left\vert\,\frac{\left(\OO\cap L_{j+1}\right)L_{j+2}}{L_{j+2}}\otimes_\Z\C\subset \frac{(K\cap L_{j+1})L_{j+2}}{L_{j+2}}\otimes_\Z\C\right.\right\}\cup\Big\{+\infty\Big\}\Bigg).
\end{equation}
\end{definition}

Torsion-free orbit depth $k$ of $\gamma$ does not exceed the orbit depth $\kappa$ of $\gamma$ and coincides with it if the abelian group $\OO/K$ has no torsion elements.
The first claim of the following Theorem is a more precise version of Theorem~\ref{thm:main minus}.

\begin{theorem} \label{main} \hfill
\begin{enumerate}
	\item[(i)] Assume that the typical fiber $\{F=t\}$ is connected. Then the first nonzero Melnikov function $M_\mu$ of any deformation \eqref{deformation} is  a combination  of base point independent iterated integrals of length not exceeding the torsion-free orbit depth $k$ of $\gamma$. The coefficients of the combination are rational functions in $t$ with poles contained in $\Sigma$.
\item[(ii)] If the orbit length $\kappa$ equals $1$ or $2$, then there exists a deformation \eqref{deformation} such that the first nonzero Melnikov function $M_\mu$ is of length $\ell(M_\mu)=k.$
\item[(ii')] If  the orbit length $\kappa$ is bigger than $1$, then there exists a deformation \eqref{deformation} such that the first nonzero Melnikov function $M_\mu$ is of length $\ell(M_\mu)=2.$
\end{enumerate}
\end{theorem}

\begin{corollary}\label{cor: d components}
	If the typical fiber $\{F=t\}$ has $d$ connected components, then the first nonzero Melnikov function $M_\mu$ of any deformation \eqref{deformation} is  a combination  of base point independent iterated integrals of length not exceeding the torsion-free orbit depth $k$ of $\gamma$. The coefficients of the combination are rational functions in $g^{-1}(t)$ for some $g\in\C[t]$, $\deg g\le d$.
\end{corollary}

\begin{remark}
	\begin{enumerate}\hfill
\item[(i)]		
	Any isomorphism $i_\alpha$ of \eqref{eq:i_alpha} respects the lower central sequence filtration as well as $\OO$ and $K$. Therefore the conditions of Theorems~\ref{thm:main minus} and ~\ref{main} are independent of the choice of the typical value $t_0$ and the base point $p_0$.
\item[(ii)] It follows from the proof of Theorem~\ref{main}(i) that the forms appearing in the iterated integrals expression for the first nonzero Melnikov function  can be chosen from  a fixed finite set of polynomial one-forms  forming a basis of $H^1(F^{-1}(t))$ for any $t\in\C\setminus\Sigma$. Then Theorem~\ref{main}(i) gives the existence of a kind of Petrov module for the first nonzero Melnikov function.
\item[(iii)] 
	In \cite{G}, the author asks for an expression of the first nonzero Melnikov function in terms of base point independent iterated integrals. The proof of Theorem~\ref{main}(ii) indicates how such an expression can be obtained.
\item[(iv)]
	We are not aware of examples with $\kappa>k$ or even with non-trivial torsion elements in  $H_1(\OO)$ of \eqref{H1}.
\item[(v)] We believe that Theorem~\ref{main}(ii)  holds for any $k\ge 1$, or, in other words, that the bound of Theorem~\ref{main} is optimal,  but the proof seems to require additional ideas. 

\end{enumerate}
\end{remark}

\begin{conjecture} \label{conj} We conjecture that for any $\gamma$ the torsion-free orbit depth $k$  and the orbit depth $\kappa$ are finite. 
\end{conjecture}

If a path $\sigma\subset F^{-1}(t_0)$ joining two points $p_0$ and $\tilde{p}_0$ lies on the fiber $F^{-1}(t_0)$,  then the isomorphism  \eqref{eq:i_alpha} can be written explicitly
	\begin{equation}\label{eq:i_sigma}
	i_\sigma:\pi_1\left(F^{-1}(t_0),p_0\right)\to\pi_1\left(F^{-1}(t_0),\tilde{p}_0\right),\, i_\sigma(\gamma)=\sigma^{-1}\gamma\sigma.
	\end{equation} 
Hence, though the isomorphism depends on the choice of  $\sigma$,   different choices of $\sigma$ lead to conjugate isomorphisms.

Define the commutative group
\begin{equation}\label{H1}
H_{1,p_0}(\OO(\gamma))=\frac{\OO_{p_0}(\gamma)}{K_{p_0}(\gamma)}.
\end{equation}
 
Note that the isomorphism $
i_\sigma^*: H_{1,p_0}(\OO(\gamma))\to H_{1,\tilde{p}_0}(\OO(i_\sigma(\gamma)))
$
induced by \eqref{eq:i_sigma} does not depend on the choice of the path $\sigma$ joining two points $p_0$ and $\tilde{p}_0$. This defines an equivalence relation and we denote the corresponding equivalence class of $H_{1,p_0}(\OO(\gamma))$  by $H_{1}(\OO)(t_0)$.
We call it \emph{the first homology group of the orbit of $\gamma$}.  A  slightly different definition of the same object was given in \cite{GI}.

It is similar to the first homology group $H_1(F^{-1}(t_0))$, but the numerator is smaller (only the orbit of $\OO$), as well as the denominator (only commutators in $K$ instead of $L_2=[\pi_1,\pi_1]$). Hence, in general, there is neither a  natural injection nor surjection between these two groups.

The key object of our studies is the vector space $\CH(t_0)=H_1(\OO)(t_0)\otimes_\Z\C$.
The first nonzero Melnikov function is independent of $p_0$ and additive, see Section~\ref{sec:Melnikov}. It is hence well-defined on $H_1(\OO)(t_0)$ and on $\C H_{1}(\OO)(t_0)$.

\begin{corollary} \label{gavrilov} 
For any deformation \eqref{deformation}, the first nonzero Melnikov function $M_\mu$ is an abelian integral if and only if $k=1$. 
\end{corollary}

\begin{proof}[Proof of the Corollary \ref{gavrilov} assuming Theorem \ref{main}]\hfill\\
If $k=1$, then for any deformation \eqref{deformation} $M_\mu$ is an abelian integral by Theorem~\ref{main}(ii).

\noindent If $k\not=1$, then by Theorem~\ref{main}(iii') there exists a deformation \eqref{deformation} such that $\ell(M_\mu)=2$, so it is not an abelian integral.
\end{proof}

\begin{remark}
The condition $k=1$ is equivalent to the injectivity of the natural mapping 
\begin{equation}\label{inj}
i:H_1(\OO)(t)\to H_1(F^{-1}(t)).
\end{equation}    
Using this last condition, Gavrilov and Iliev \cite{GI} proved the direct implication in Corollary~\ref{gavrilov}.
\end{remark}

\begin{corollary} (Ilyashenko, Fran\c coise)\label{cor:IlyaFranc}\hfill\\ Consider a polynomial $F\in\mathbb{C}^2$ verifying the following generic properties:
\begin{enumerate}
\item[(G1)] All critical points of $F$ are of Morse type and have disctinct critical values.
\item[(G2)] The fibers intersect the line at infinity transversally. 
\end{enumerate}
Consider any deformation \eqref{deformation} of the foliation $dF=0$ and the first nonzero Melnikov function $M_\mu(\gamma)$ of the displacement function along a cycle $\gamma$ of $F$. Then $M_\mu(\gamma)$ is an abelian integral. 
\end{corollary}

\begin{proof} Ilyashenko \cite{I}, proves that the hypothesis implies that $\OO=\pi_1$, so $K=L_2=L_2\cap\OO$. Hence, $\kappa=1$ is the smallest $\kappa\geq1$ verifying \eqref{kappa}. 
By Theorem \ref{main}, for any deformation $\eta$, the first nonzero Melnikov function is an iterated integral of length $1$, i.e. an abelian integral. 
\end{proof}
\begin{remark}\label{cond*}
The above Corollary is a consequence of results of Ilyashenko \cite{I} and Fran\c coise \cite{F}. 
Indeed, Ilyashenko  first proves that the above hypothesis implies that $\OO=\pi_1$. Next he proves that the following condition is verified.
\begin{enumerate}
\item[(*)]	
Any form $\omega$ such that  $\int_\gamma\omega$ vanishes identically is  \emph{relatively exact} i.e. can be written as $\omega=gdF+dR$, for some polynomials $g$ and $R$. 
\end{enumerate}
This condition is now known as the \emph{condition (*) of Fran\c coise} \cite{F}. Condition (*) of Fran\c coise assures that Fran\c coise algorithm works. The algorithm itself then produces a polynomial form $\omega_k$ such that $M_\mu(\gamma)=\int_\gamma\omega_k$, so $M_\mu(\gamma)$ is an abelian integral. 
\end{remark}

\section{Sketch of the proof of Theorem \ref{main}}\label{sec:sketch}

The proof consists of several parts. 
First, we fix a typical fiber $F^{-1}(t_0)$ and study the corresponding  first homology group $\C H_{1}(\OO)(t_0)$ of the orbit of $\gamma$. We show  that the  first nonzero Melnikov function $M_\mu$ defines	 a linear functional $\mathbf{M}_{\mu,t_0}\in\left(\C H_{1}(\OO)(t_0)\right)^*$, $\mathbf{M}_{\mu,t_0}(\gamma(t_0))=M_\mu(t_0)$. Moreover, the assumptions of Theorem~\ref{main} imply that it vanishes on $\OO\cap L_{k+1}$.

We want to realize the dual vector space $(\C H_{1}(\OO)(t_0))^*$ as the space of linear combinations of iterated integrals.

Let $\C_k$ be the space of jets of degree $\le k$ of formal power series in non-commuting variables $X_1,...,X_n$, where $n=\dim H_1(F^{-1}(t))$.
Chen (\cite{C} and \cite{H}) constructs an injective  \emph{multplicative} homomorphism $II_k:\,\pi_1/L_{k+1}\to (\C_k)^*$ given by a power series whose coefficients are iterated integrals. 

Using Baker-Campbell-Hausdorff formula \cite{J}, we prove that $\log\circ II_k$ induces a \emph{linear} isomorphism from $\C H_{1}(\OO)(t_0)$ to  a factor space $N(t_0)$ of a subspace of $\C_k$.

Next, we prove that for any $\phi\in (N(t_0))^*$ the mapping  $\phi\circ\log\circ II_k\in (\C H_{1}(\OO)(t_0))^*$ is a combination of base point independent  iterated integrals. This proves that $\mathbf{M}_{\mu,t_0}$ is a linear combination of base point independent  iterated integrals of length at most $k$.

Finally,  we consider two isomorphic vector bundles  $\mathcal{H}=\cup_{t_0\in\C\setminus \Sigma} \C H_1(\OO)(t_0)$ and $\mathcal{N}=\cup_{t_0\in\C\setminus \Sigma} N(t_0)$ and their respective sections $\mathbf{M}_\mu$ and $s=\left((\log\circ II_K)^*\right)^{-1}(\mathbf{M}_\mu)$ defined fiberwise by $\mathbf{M}_{\mu,t_0}$.  The above base point independence of iterated integrals implies that  the  vector bundle $\mathcal{N}$ is well defined.
Note that the function $M_\mu(t)=\mathbf{M}_{\mu}(\gamma(t))$ is   multivalued, being evaluation of the univalued section  $\mathbf{M}_{\mu}$ against  the multivalued section $\gamma(t)$.

We observe that  $\mathcal{N}$ is obtained by operations with regular subbundles of the trivial bundle $\C_k\times \left(\C\setminus \Sigma\right)$. Moreover, as the base of $\mathcal{N}^*$ is non-compact, it is trivial \cite[Theorem 30.4]{Forster}.
Therefore, there is  a basis of sections $\{f_j\}$ of $\mathcal{N}^*$  consisting of linear combinations with rational in $t$ coefficients of linear functionals on $\C_k$ (the latter ones are just coefficients of monomials of a jet). As $s$ is of regular growth, it is a linear combination with rational in $t$ coefficients of $f_j$. 

Pushing back by $(\log\circ II_K)^*$, we see that $\mathbf{M}_\mu=(\log\circ II_K)^*(s)$ is a linear combination with rational in $t$ coefficients of sections $(\log\circ II_K)^*(f_j)$, which are  base point independent iterated integrals of length at most $k$. 

Points (ii) and (ii') are proved by direct calculation using Fran\c coise's algorithm of \cite{F, G}, see Section~\ref{sec: lower bounds proofs}.

\section{Basic properties of the first nonzero Melnikov function on a fixed fiber.}\label{sec:Melnikov}

Let $t_0\in\C\setminus \Sigma$ be a typical  value of $F$.
\begin{proposition}\label{prop:Melnikof functional}
	The first nonzero Melnikov function $M_\mu$ is base point independent and
	 defines a linear functional $\mathbf{M}_{\mu,t_0}$ on $\C H_1(\OO)(t_0)$.

Let $k\ge 1$ be  the torsion-free orbit depth of $\gamma$. Then  $\mathbf{M}_{\mu,t_0}$ vanishes on $\OO\cap L_{k+1}$.
\end{proposition}

\begin{lemma}\label{lem:Mk is additive}
Let $P_1,P_2$ be two germs of the form
$$
P_j(t,\epsilon)=t+\sum_{i=\mu_j} M_{j,i}(t)\epsilon^i, \quad M_{j,\mu_j}\not\equiv 0.
$$ 
Then the following holds
\begin{itemize}
\item[(i)] $P_1\circ P_2=P_2\circ P_1 \left(\operatorname{mod}\epsilon^{\mu+1}\right)$, where $\mu=\max\left\{\mu_1,\mu_2\right\}$.
\item[(ii)]	In particular, $P_2^{-1}\circ P_1\circ P_2$ and $P_1$ coincide modulo $\epsilon^{\mu_1+1}$.
\item[(iii)]	If $\mu_1=\mu\le\mu_2$ then 
$$
P_1\circ P_2=t+\left(M_{1,\mu}+M_{2,\mu}\right)\epsilon^\mu+o(\epsilon^\mu).$$
	\end{itemize}
\end{lemma}
\begin{proof} Direct computation.\end{proof}

\begin{lemma}\label{Kinv}
All Melnikov functions of order $<\mu$ of elements of $\OO$ vanish identically. 
\end{lemma}
\begin{proof}
By analytic continuation along $\alpha$, all Melnikov functions of order $<\mu$ of $Mon_\alpha (\gamma)$ vanish for any $\alpha$. For their conjugates and products the claim follows from Lemma~\ref{lem:Mk is additive} applied to respective first return maps $P_\gamma(\epsilon,t)=t+\Delta(\epsilon,t)$.
\end{proof}

\begin{proof}[Proof of Proposition~\ref{prop:Melnikof functional}]
	By Lemma~\ref{lem:Mk is additive}(ii), $M_\mu$ vanishes on $K$ and is base point independent. Therefore, by Lemma~\ref{lem:Mk is additive}(iii), $M_\mu$ defines an additive  homomorphism from $H_1(\OO)(t_0) $ to the linear space of germs of holomorphic functions at $t_0$, and thus descends to a linear functional on $\C H_1(\OO)(t_0)$.
	
	If \eqref{k} holds, then any element of $\OO\cap L_{k+1}$ is a product of elements corresponding to torsion elements of $H_1(\OO)$ and of elements of $K$. Since  $\mathbf{M}_{\mu,t_0}$ is additive, it vanishes on each one of them as well as on their products.
\end{proof}

\section{Direct sum decomposition of $\CH$}\label{sec:CH structure}

Filtration of $\pi_1$ by $L_i$ induces filtrations on both $\OO$ and $K$ which define a filtration of $\CH$. 
In this section we calculate  it explicitly.

Denote 
\begin{equation}
\OO_i=\frac{K\left(\OO\cap L_i\right)}{K}\otimes_\Z\C.
\end{equation}
Then, by \eqref{k}, $\OO_{k+1}=0$ and we have the  filtration
$$
0\subset \OO_k\subset \OO_{k-1}\subset...\subset \OO_1=\CH,
$$
and, therefore,

\begin{equation}\label{eq:ch1decomposition}
\CH\cong\bigoplus_{i=1}^k\frac{\OO_{i}}{\OO_{i+1}}.
\end{equation}

We have  well-defined mappings $\phi_i$
\begin{equation}\label{eq:phi_k}
\OO_i\cong\frac{\OO\cap L_i}{K\cap L_{i}}\otimes_\Z\C\xrightarrow{\phi_i} \text{Im}\, \phi_i=\frac{(\OO\cap L_i)L_{i+1}}{(K\cap L_{i})L_{i+1}}\otimes_\Z\C,
\end{equation}
and the epimorphisms $\psi_i$
\begin{equation}\label{eq:epimorphism}
\frac{L_i}{L_{i+1}}\otimes_\Z\C\xrightarrow{\psi_i}\frac{L_i}{ \left(K\cap L_{i}\right)L_{i+1}}\otimes_\Z\C.
\end{equation}
Note that $\text{Im}\,\phi_i\subseteq \text{Im}\,\psi_i$.

\begin{lemma}\label{lem:ker phi_i}
	$\ker \phi_i=\OO_{i+1}$. Therefore $\frac{\OO_{i}}{\OO_{i+1}}\cong \operatorname{Im}\phi_i$.
\end{lemma}

\begin{corollary}\label{cor:k'}
	If Conjecture~\ref{conj} holds for some $k$, then it holds for all $k'>k$. 
\end{corollary}
Indeed, in this case $\ker\phi_{k}=\{e\}$ is trivial, so $\phi_{k+1}$ is a mapping from the trivial group, so $\ker\phi_{k+1}=\{e\}$ as well, and so on.

\section{Iterated integrals generate   $\left(\C H_1(\OO)(t_0)\right)^*$}\label{sec:Iterated Integrals}

We want to find combinations of iterated integrals forming a  basis of  $\left(\C H_1(\OO)(t_0)\right)^*$. 
\subsection{Basic definitions and properties}
Let $\omega_i=\phi_i\,dt$ be one-forms on $[0,1]$. We define the iterated integral as 
\begin{equation}\label{eq:definition of iterated integrals 1}
\int_0^1\omega_1\dots\omega_n=\int_0^1\left(\int_0^{t_n}\left(\dots\left(\int_0^{t_3}\left(\int_0^{t_2}\omega_1\right)\omega_2\right)\dots\right)\omega_{n-1}\right)\omega_n,
\end{equation}
or, equivalently,
\begin{equation}\label{eq:definition of iterated integrals 2}
\int_0^1\omega_1\dots\omega_n=\int_{\{0\le t_1\le...\le t_n\le 1\}}\phi_1(t_1)\dots\phi_n(t_n) \,dt_1\dots dt_n.
\end{equation}
For $n=1$ we get the usual abelian integral.
\begin{definition}
	We define the \textit{formal length} of \eqref{eq:definition of iterated integrals 1} to be $n$, and the formal length of a  linear combination of iterated integrals  to be the maximal formal length of its terms. For function representable as  a linear combination of iterated integrals we define its length as a minimal formal length of such representation.
\end{definition}
For a path $\gamma:[0,1]\to F^{-1}(t_0)$ and one-forms $\omega_i\in\Omega^1(F^{-1}(t_0))$ we define $\int_{\gamma}\omega_1\dots\omega_n=\int_0^1\gamma^*\omega_1\dots\gamma^*\omega_n$. If the forms $\omega_i$ are holomorphic, then the homotopic deformation with fixed endpoints of $\gamma$ does not change the integral.

In general, iterated integrals are not additive for $n>1$. However,  for a product of two paths $\gamma=\gamma_1\gamma_2$ the following shuffling formula holds:
\begin{equation}\label{eq:shuffling formula}
\int_{\gamma}\omega_1\dots\omega_n=\sum_{i=0}^n\int_{\gamma_1}\omega_1\dots\omega_i\int_{\gamma_2}\omega_{i+1}\dots\omega_n.
\end{equation}
This formula is equivalent to the multiplicativity of the Chen map $II$ defined below. It implies that for $n>1$ and for a loop $\gamma$ the iterated  integral in general depends on the choice of the base point of $\gamma$.

As an important corollary of \eqref{eq:shuffling formula}, one gets the fundamental fact  that iterated integrals of length $k$ vanish on $L_{k+1}$ and are additive on the free abelian group $L_k/L_{k+1}$.
The fundamental theorem of Chen claims that the  iterated integral of length $k$ generate $\left(L_k/L_{k+1}\otimes_{\Z}\C\right)^*$, see \cite{H}.

\subsection{Chen's iterated integrals morphism.}
Recall the construction of Chen  (see \cite{C, H}). Let $\{\eta_1,...,\eta_n\}$ be a tuple of polynomial one-forms in $\C^2$ forming a basis of  $H^1(F^{-1}(t))$, for all $t\in\C\setminus \Sigma$. Choose a base point $p_0\in F^{-1}(t_0)$, and let 
$\{\gamma_i\}_{i=1}^n\subset \pi_1(F^{-1}(t_0), p_0)$ form the  basis of $H_1(F^{-1}(t_0))$ dual to the basis  $\{\eta_1,...,\eta_n\}$.

Recall  that $\pi_1(F^{-1}(t_0), p_0)$ is a free group generated by $\{\gamma_i\}_{i=1}^n$.

Let $\tilde{X}\xrightarrow{pr} F^{-1}(t_0)$ be the universal cover of $F^{-1}(t_0)$, $pr(\tilde{p}_0)=p_0$. 
Following Harris \cite{H}, we denote by $\mathfrak{g}$ the free Lie algebra generated by formal variables $X_1, ..., X_n$, $\mathfrak{g}\subset U(\mathfrak{g})=\C[X_1,...,X_n]$ -- a free associative algebra  on   $X_1,...,X_n$. Let  $G=\exp\mathfrak{g}$ be the Lie group corresponding to $\mathfrak{g}$ and lying in the formal completion $\hat{U}(\mathfrak{g})$.

Consider the $\mathfrak{g}$-valued form $\sum_{i=1}^n\eta_i\otimes X_i$ on $\tilde{X}$.
The horizontal section of the Chen connection $\nabla s=ds-s\sum_{i=1}^n\eta_i\otimes X_i$ on the bundle $\tilde{X}\times G \to \tilde{X}$  passing through the point $\left(\tilde{p}_0, 1\right)\in \tilde{X}\times G $ is the graph of Chen's \emph{iterated integrals mapping} $II:\tilde{X}\to G$.

An explicit formula for $II$ is 
\begin{equation}\label{II}
II(\tilde{p})=1+x,\,\text{where}\, x=\sum_{(i_1,...,i_\ell)}\left(
\int_{\gamma}\eta_{i_1}...\eta_{i_\ell}
\right)X_{i_1}...X_{i_\ell}
\end{equation}
and $\gamma$ is a path from $\tilde{p}_0$ to $\tilde{p}$. Restricting to $pr^{-1}(p_0)$, this gives a mapping $II:\pi_1(F^{-1}(t), p_0)\to G$. The invariance of $\sum_{i=1}^n\eta_i\otimes X_i$ with respect to left multiplication by $G$ implies that this mapping is a homomorphism.

Let $k$ be as in Theorem~\ref{main}(ii). Let $\mathfrak{m}\subset\C[X_1,...,X_n]$ be the maximal ideal generated by $X_1,...,X_n$ and denote $\C_k=\C[X_1,...,X_n]/\mathfrak{m}^{k+1}$.
Consider the   $k$-th jet $II_k:\pi_1(F^{-1}(t),p_0)\rightarrow G_k=G/\mathfrak{m}^{k+1}\subset\C_{k}$ of $II$
\begin{equation}\label{IIk}
 II_k(\gamma)=1+x_k, \quad  x_k=\sum_{(i_1,...,i_\ell), \ell\le k}\left(
 \int_{\gamma}\eta_{i_1}...\eta_{i_\ell}
 \right)X_{i_1}...X_{i_\ell}.
\end{equation}
Consider moreover the composition 
$$
\log\circ II_k: \pi_1(F^{-1}(t), p_0)\to \mathfrak{g}/\mathfrak{m}^{k+1}\subset\C_{k},
$$
\begin{equation}
\log(II_k(\gamma))=x_k-x_k^2/2+x_k^3/3-...=\sum_{\|\alpha\|\leq k}
q_\alpha(\gamma)X^\alpha+\mathfrak{ m}^{k+1},
\end{equation}
where $q_\alpha(\gamma)$ are some polynomials in
the coordinates of $ II_k(\gamma)$, i.e.  polynomials in
$\int_{\gamma}\eta_{i_1}...\eta_{i_\ell}$.
Chen's theorem implies  that $II$ is an injective homomorphism of groups
and  that $\log\circ II_k(\pi_1/L_{k+1})$ spans $\mathfrak{g}/\mathfrak{m}^{k+1}$, see \cite{H}.

\begin{remark}\label{rem:lincomb}
	 Products of iterated
integrals are given by linear combinations of iterated
integrals whose  length is bounded by the sum of their lengths. Hence,  the polynomial
$q_\alpha(\gamma)$ can also be written
as linear combinations of iterated integrals of length at most $\ell\le k$.
\\
\end{remark}

Define $$N_0:=\Span\{\log(II_k(K))\}\subset N_1:=\Span\{\log(II_k(\OO))\}$$ and consider the quotient $N_{p_0}=\frac{N_1}{N_0}$.
\\

\begin{proposition}\label{prop:isomorphism}
The mapping 
\begin{equation}\label{eq:log IIk is isomorphism}
\log\circ II_k:\pi_1/L_{k+1}\to \C_k
\end{equation} induces a linear  isomorphism of 	$\CH$ and $N_{p_0}$.
\end{proposition}
\begin{proof}
	The induced mapping $\log\circ II_k :\CH\to N_{p_0}$  is well-defined and linear by  Baker-Campbell-Hausdorff formula. Indeed,  if $\sigma_1, \sigma_2\in\OO$, then
	$$\log (II_k(\sigma_1\sigma_2))=\log(II_k(\sigma_1))+\log(II_k(\sigma_2))+\sum \log (\delta_\alpha),
	$$
	where $\delta_\alpha$ are various group commutators of $II_k(\sigma_1), II_k(\sigma_2)$. But $II_k$ is a homomorphism, so $\delta_\alpha$ are images under $II_k$ of various commutators of $\sigma_1,\sigma_2$, i.e. elements of $K$. Therefore, all $\log (\delta_\alpha)$ are in $N_0$. Moreover, if $\sigma_2\in K$,
	then  $\log(II_k(\sigma_2))$ is also in $N_0$. This proves that the mapping is well-defined and linear.
	Since it is surjective by definition, we only need to prove injectivity.

	Let $\sigma\not\in K$ and assume that $\sigma\in\OO_i\setminus\OO_{i+1}$. Choose $\delta\in\frac{L_i}{L_{i+1}}\otimes_\Z\C$ such that  $\phi_i(\sigma)=\psi_i(\delta)\not=0\in \frac{L_i}{ \left(K\cap L_{i}\right)L_{i+1}}\otimes_\Z\C$, where $\phi_i, \psi_i$ are defined in \eqref{eq:phi_k},\eqref{eq:epimorphism} respectively. In particular,  $\delta\not\in \left(K\cap L_i\right)L_{i+1}$.
	
	In the sequel of the proof we work $\,(\operatorname{mod}\mathfrak{m}^{i+1})$ without mentioning this explicitly.
	By Chen's theorem, the iterated integrals of length $i$ generate  $\left(\frac{L_i}{L_{i+1}}\otimes_\Z\C\right)^*$, so
	$$
	II_k\left(\frac{L_i}{L_{i+1}}\otimes_\Z\C\right)=1+P_i,
	$$ 
	$$
	II_k\left(\frac{K\cap L_i}{L_{i+1}}\otimes_\Z\C\right)=1+R_i,
	$$
	where $P_i$ is the set of all homogeneous Lie polynomials in $X_j$ of degree exactly $i$ and $R_i\subset P_i$. Therefore, 	$ II_k(\delta)=1+p_i(X)$, $p_i\not\in R_i$ and $\log\circ II_k(\delta)=p_i\not\in \log(1+R_i)=R_i$. 
	\begin{lemma}
		$N_0\cap P_i = R_i$.
	\end{lemma}
	\begin{proof}
		Consider the mapping $\log\circ II_j:L_j/L_{j+1}\otimes_{\Z}\C\to \mathfrak{m}^j/\mathfrak{m}^{j+1}$ . Chen's theorem implies that this is an injective linear homomorphism, so $\log\circ II_j(L_j/L_{j+1})$ is a lattice in $\mathfrak{m}^j/\mathfrak{m}^{j+1}$ (recall that $\frac{ L_j}{L_{j+1}}$ is a free abelian group) and therefore, $\log\circ II_j(\frac{K\cap L_j}{L_{j+1}})$ is also a lattice. 
		
		Now, assume 
		\begin{equation}\label{eq:representation of p}
		p=\sum \lambda_{\alpha} v_\alpha\in \mathfrak{m}^i,
		\end{equation} 
		where $v_\alpha= \log\circ II_i(\delta_\alpha)$, for some $ \delta_\alpha\in K$, and let $j<i$ be the smallest integer such that not all $v_\alpha$ lie in $\mathfrak{m}^{j+1}$ or, equivalently, not all $\delta_{\alpha}\in L_{j+1}$. Then, replacing  $v_\alpha$ by their linear combinations with \textit{integer coefficients} $w_{\beta}=\sum_{\alpha}c_{\beta}^{\alpha}v_\alpha$, $c_{\beta}^{\alpha}\in\Z$, we can  
		assume that those $w_\beta$ which are not in $\mathfrak{m}^{j+1}$, are linearly independent over $\Z$ in $ \mathfrak{m}^j/\mathfrak{m}^{j+1}$. But then they are linearly independent over $\C$, as they lie in the lattice $\log\circ II_j(\frac{K\cap L_j}{L_{j+1}})$. Therefore all $w_\beta\in \mathfrak{m}^{j+1}$.
		
		But, by Baker-Campbell-Hausdorff formula, 
		$$
		w_\beta= \log\circ II_i(\prod_{\alpha}\delta_\alpha^{c_{\beta}^{\alpha}})+\sum_r\log\circ II_i (\tilde{\delta}_r),
		$$ 
		where $\tilde{\delta}_r\in K\cap L_{j+1}$ are various commutators of $\delta_\alpha$. So we transformed \eqref{eq:representation of p}  to another representation of $p$ with bigger $j$. Repeating this several times, we get a representation with $v_\alpha\in  \in \mathfrak{m}^i$. This means that $p\in P_i$.
	\end{proof}
	
	 Hence, $\log\circ II_k(\delta)\not=0$ in $N_{p_0}$, thus proving injectivity of $\log\circ II_k :\CH\to N_{p_0}$.\end{proof}

\subsection{Base-point independence of linear functionals on $\CH$}

\begin{proposition}\label{prop:base point independence}\hfill
	\begin{enumerate}
		\item[(i)] The subspaces $N_0=N_0(p_0), N_1=N_1(p_0)\subset \C_k$ do not depend on the choice of the base point $p_0\in F^{-1}(t_0)$.
		\item[(ii)] The map $N_1(p_0)\to N_1(p_1)$ induced by the change of the base point descends to the identity map in the quotient $N(t_0) =N_1/N_0$. 
	\end{enumerate}
\end{proposition}

\begin{lemma}\label{lem:alpha}
	For any $\gamma\in\OO$ and any, not necessarily closed, path $\alpha$
\begin{equation}\label{eq:BPI}
\log\circ II(\alpha^{-1}\gamma\alpha)=\log\circ II(\gamma) \,\left(\operatorname{mod}N_0\right).
\end{equation}
\end{lemma}

This will follow from two lemmas 
\begin{lemma}\label{lem:N_0=}
	$N_0=\left[N_1, \mathfrak{g}\right].$
\end{lemma}
\begin{lemma}\label{lem:BCH}
	If $\log\circ II(\alpha)\in \mathfrak{g}$, $\log\circ II(\gamma)\in N_1$, then
	$$\log\circ II(\alpha^{-1}\gamma\alpha)-\log\circ II(\gamma)\in\left[N_1, \mathfrak{g}\right].$$
\end{lemma}

\begin{proof}[Proof of Lemma~\ref{lem:alpha}]
Chen's construction implies $II(\alpha)\in G$. Hence, $\log\circ II(\alpha)\in\mathfrak{g}$ and Lemma~\ref{lem:alpha} follows from the above Lemmas.
\end{proof}

\begin{proof}[Proof of Lemma~\ref{lem:N_0=}]
	Let $\gamma\in\OO, \sigma\in\pi_1\left(F^{-1}(t),p_0\right)$. 
	Then, as $II$ is a homomorphism, we have
	$$
	\log\circ II([\gamma, \sigma])=\left[\log\circ II(\gamma),\log\circ II(\sigma)\right]+...,
	$$
	where the dots denote some higher order commutators of $\log\circ II(\gamma)$ and $\log\circ II(\sigma)$.
	
	Standard associated graded algebra arguments show that it is enough to prove that 
	$$
	\left\langle \log\circ II(\sigma),\, \sigma\in\pi_1\left(F^{-1}(t),p_0\right)\right\rangle=\mathfrak{g},$$
	which is a direct consequence of Chen's theorem.
	\edz{Check "standard ass. grad. algebra arguments"}
%
		\end{proof}

\begin{proof}[Proof of Lemma~\ref{lem:BCH}]
	By multiplicativity of $II$ and by the Baker-Campbell-Hausdorff formula,
	$$
	\log\circ II(\alpha^{-1}\gamma\alpha)=\log\circ II(\gamma)+...,
	$$
	where dots denote higher order commutators of $\log\circ II(\alpha)$ and $\log\circ II(\gamma)$. Each of these commutators necessarily contains at least one $\log\circ II(\gamma)\in N_1$ and at least one of $\log\circ II(\alpha)\in\mathfrak{g}$. But $[N_1,\mathfrak{g}]=N_0\subset N_1$, so each of these commutators is in $[N_1,\mathfrak{g}]=N_0$.
\end{proof}

\begin{proof}[Proof of Proposition~\ref{prop:base point independence}]
	Choosing another base point $p_1$ results in replacement of $$N_1(p_0)=\Span\{\log(II_k(\gamma)), \gamma\in\OO_{p_0}\}	$$
	by $$N_1(p_1)=\Span\{\log(II_k(\alpha^{-1}\gamma\alpha)), \gamma\in\OO_{p_0}\}	,$$ where $\alpha$ is some path joining $p_0$ and $p_1$.
	By Lemma~\ref{lem:alpha}, the difference is in $N_0(p_0)\subset N_1(p_0)$, so $N_1(p_0)=N_1(p_1)$. Similarly, $N_0(p_0)=N_0(p_1)$.
	
	The second claim then follows directly from Lemma~\ref{lem:alpha}.
	\end{proof}

\section{Vector bundles $\mathcal{H,C,N}$ and proofs of Theorems~\ref{thm:main minus} and~\ref{main}} \label{sec:vector bundles}
Here we investigate the  dependence of previous constructions on $t$. The fundamental fact is the following 
\begin{lemma}\label{lem:moderate growth of II}
	Let $\eta_i$ be polynomial one-forms on $\C^2$, and let $I(t,p_0)=\int_{\gamma(t)}\eta_1\dots\eta_n$ be an iterated integral of length $n$ over a family of loops $\gamma(t)\subset \pi_1\left(F^{-1}(t),p(t)\right)$.
	Then $I(t)$ is a multivalued function of regular growth on $\left(t,p(t)\right)\in \left(\C\setminus\Sigma\right)\times\left(\C^2\setminus F^{-1}(\Sigma)\right)$. 
\end{lemma}
This can be proved along the lines of \cite{M}: after a resolution of singularities of the fibration $F:\C P^2\to\C P^1$, the family  $\gamma(t)$  can be cut, by some analytic transversals,  into a union of several pieces, 
each one lying in its own chart. In each chart, $F$ is a monomial, and $\eta_i$ are meromorphic with poles on the preimage of the union of atypical fibers. As $t$ tends to an atypical value, the iterated integrals along these pieces can be easily polynomially bounded from above, so the total integral $I(t)$  has moderate growth by the shuffling formula \eqref{eq:shuffling formula}.

We consider several vector bundles over $\C\setminus\Sigma$.

The union of all $\CH(t_0)$ over all  $t_0\in\C\setminus \Sigma$ of typical values of $F$ forms a vector bundle $\mathcal{H}$.
Taken together, the Melnikov functionals  $\mathbf{M}_{\mu,t_0}$, $t_0\in\C\setminus\Sigma$, define a section $\mathbf{M}_\mu$ of its dual bundle $\mathcal{H}^*$.

The bundle $\mathcal{H}$ inherits the linear Gauss-Manin connection from the homotopy fibration. Evaluation of $\mathbf{M}_\mu$ on its horizontal section $\gamma(t)$ gives the Melnikov function:
$$
M_\mu(t)=\mathbf{M}_\mu(\gamma(t)).
$$

By Proposition~\ref{prop:isomorphism} we can identify the vector bundle $\mathcal{H}^*$ with the vector bundle $\mathcal{N}^*$, where $\mathcal{N}$ is the vector bundle 
$$
\mathcal{N}=\left\{\cup_{t\in \C\setminus \Sigma}N(t)\to \C\setminus \Sigma\right\}.
$$
 The bundle $\mathcal{N}$ is obtained from the trivial bundle 
$$
\mathcal{C}_k=\left\{\C_k\times \left(\C\setminus \Sigma\right)\to \C\setminus \Sigma\right\},
$$ 
by taking the subbundle 
$$
\mathcal{N}_1=\left\{\cup_{t\in \C\setminus \Sigma}N_{1}(t)\to \C\setminus \Sigma\right\},
$$
and then taking the quotient by the subbundle
$$
\mathcal{N}_0=\left\{\cup_{t\in \C\setminus \Sigma}N_{0}(t)\to \C\setminus \Sigma\right\}.
$$
By Proposition~\ref{prop:base point independence}, all these bundles are well defined, i.e. do not depend on a particular choice of the base points.

\begin{lemma}\label{lem:bperp}
	Assume that a subbundle $\mathcal{B}$ of $\mathcal{C}_k$ is spanned by sections with moderate growth as $t$ tends to $\Sigma\cup\{\infty\}$. Then  the subbundle $\mathcal{B}^\perp$ of  $\mathcal{C}_k^*$  is spanned by sections of the form 
	\begin{equation}\label{eq:basis of Bperp}
	\tilde{f}_j=\sum_{\|\alpha\|\le k}b_\alpha^j(t)c_\alpha,
	\end{equation}
	where $b_\alpha^j(t)$ are rational functions of $t$, and $c_\alpha$ are the standard linear functionals on $\C_k$ given by taking the coefficients of the monomials $X^\alpha$ of a jet in $\C_k$.
\end{lemma}
\begin{proof}
	The vector bundle $\mathcal{B}$ defines  a univalued mapping $\C\setminus\Sigma\to \mathrm{Gr}(\C_k, d)$ of  the base $\C\setminus\Sigma$ to the Grassmanian of all $d$-dimensional subspaces of $\C_k$, where $d$ is the dimension of the fibers $B(t)$  of $\mathcal{B}$.
	This mapping has regular singularities at  $\Sigma\cup\{\infty\}$ by assumption, so is a rational map. Equivalently, the coefficients of  the equations defining the fiber $B(t)$ are rational functions of $t$. This is exactly the claim of the Lemma.
	\end{proof}

\begin{corollary}\label{cor:sections of N*}
There exists rational functions $b_\alpha^j(t)$ with poles in $\Sigma$ such that the sections 
\begin{equation}\label{eq:basis of N*}
f_j=\sum_{\|\alpha\|\le k}b_\alpha^j(t)c_\alpha
\end{equation}
of $\mathcal{C}_k^*$ define a basis of sections of $\mathcal{N}^*$, where $c_\alpha$ are the standard linear functionals on $\C_k$ as in Lemma~\ref{lem:bperp}.
\end{corollary} 
\begin{proof}
 Choose a subset of sections $\tilde{f}_j$ of $\mathcal{N}_0^\perp$ as in Lemma~\ref{lem:bperp} defining a basis of $\left(N_1(t)/N_0(t)\right)^*$ at a generic $t\in \C\setminus\Sigma$. Suitable combinations $f_j$ of these sections with rational in $t$ coefficients will then define a basis, for all $t\in \C\setminus\Sigma$. In particular, in representation \eqref{eq:basis of N*} the poles of the coefficients $b_\alpha^j(t)$ will be in $\Sigma$.
\end{proof}

We will need later  the following Corollary, which is a version of the algebraic de Rham theorem.
\begin{corollary}\label{cor:n1perp}
	Any cohomology class  $\eta\in H^1(F^{-1}(t_0))$ vanishing on $\OO L_2/L_2\subset  H^1(F^{-1}(t_0)) $ is a restriction of a rational one-form $E\in\Omega^1(\C^2)$ whose integral vanishes identically on $\gamma$.
\end{corollary}
\begin{proof}
	For  $k=1$, we have  $\log\circ II_1(\gamma)=\sum \int_{\gamma}\eta_i\,X_i$, which identifies $\C_1\cong\C\oplus H_1(F^{-1}(t_0))$ and, therefore, $\C_1^*\cong\C\oplus H^1(F^{-1}(t_0))$. Let $\mathcal{B}=\mathcal{N}_1$ and let 
	$$
	f=a_0(t)\cdot\mathbbm{1}+\sum a_i(t)c_i,
	$$be a linear combination of sections \eqref{eq:basis of Bperp} such that $f(t_0)$ coincide with $(0,\eta)$.  Here $a_i(t)$ are rational functions of $t$. Then one can take $E=\sum a_i(F)\eta_i$.
	\end{proof}
\begin{proof}[Proof of Theorem~\ref{main}]	
Now, the section $s=\left((\log\circ II_K)^*\right)^{-1}(\mathbf{M}_\mu)$ also has regular growth, so 
$$
s=\sum b_j^M(t)f_j,
$$
where  $b_\alpha^M(t)$ are also rational functions of $t$.

Therefore
\begin{equation}\label{eq:final expression for Mmu}
M(\gamma(t))=s_M(\log\circ II_k (\gamma(t))=\sum_j b_j^M(t)g_j(\gamma(t)), 
\end{equation}
where  $g_j(\gamma)=f_j(\log\circ II_k (\gamma))$ are combinations with rational in $t$ coefficients of iterated integrals of $\gamma$, by Remark~\ref{rem:lincomb}. Moreover, $g_j$ are base point independent by Proposition~\ref{prop:base point independence}.
\end{proof}
\begin{proof}[Proof of Theorem~\ref{thm:main minus}]
	Theorem~\ref{thm:main minus} follows from Theorem~\ref{main} as the torsion-free orbit length $k$ of $\gamma$ does not exceed the orbit length $\kappa$ of $\gamma$.
\end{proof}
\begin{proof}[Proof of Corollary~\ref{cor: d components}]	
If a typical fiber $\{F=t\}$ has $d$ components, then $F=g(H)$, where $H\in\C[x,y]$ and $g\in\C[s],\deg g=d$. Applying 	Theorem~\ref{main} to the foliation $\{H=s\}$, we see that $M_\mu$ has the form \eqref{eq:final expression for Mmu} with coefficients being rational functions of $s=g^{-1}(t)$, as required.
\end{proof}

\section{Alternative proof of Theorem \ref{main} using Chen's bordered determinant}

In this section we sketch an alternative proof of Theorem \ref{main} (i)(ii) inspired by the proof of Gavrilov and Iliev \cite{GI} in the abelian integral case. We first recall the main ideas of their proof here.
\subsection{Proof of Gavrilov and Iliev in the abelian integrals case}
In \cite{GI} the authors construct a determinant given by the de Rham pairing of a basis of the homology $H_1(\OO)$ and its dual space $H^1(\OO)$. The hypothesis of the injectivity of \eqref{gavrilov} assures, by the de Rham theorem, that this dual space is given by integration of polynomial one-forms.  The pairing is given by integrating the forms along the corresponding cycles. The determinant is then bordered by adding one row and one column. The row corresponds to the cycle $\gamma$ along which one considers the dispacement function and its first nonzero Melnikov function. 
The column corresponds to the first nonzero Melnikov function for the corresponding cycle. 
We call the above determinant the \emph{bordered de Rham determinant}. 

The determinant of the bordered de Rham matrix vanishes due to its size. Developping it with respect to the last row, they obtain an expression for the first nonzero Melnikov function $M_\mu(\gamma)$ as a linear combination of abelian integrals. They study the monodromy of the coefficients and show that they are univalued. It then follows from growth estimates that they are rational functions. They then conclude that the first nonzero Melnikov function is an abelian integral. 

\subsection{Our proof in the case of iterated integrals of any order} In our case, working with simple (abelian) integrals is not enough, since many minors in the bordered de Rham matrix vanish. We hence work with iterated integrals instead of just simple integrals. Instead of the de Rham pairing, we have the Chen pairing given by iterated integration. 

The new problem that we encounter is that contrary to the situation with abelian integrals which behave linearly with respect to monodromy, this is not the case with iterated integrals due to the shuffling formula \eqref{eq:shuffling formula}. 
We overcome this problem by  working in the quotient space by the commutator group $K$. 

\subsection{Bordered Chen pairing matrix}
We first construct the bordered Chen pairing matrix adapted to our use.  It is a matrix $C_\gamma$, whose rows correspond to  cycles forming a basis of $H_1(\OO)$ and its columns correspond to a basis of a dual space $H^1(\OO)$. By Proposition~\ref{prop:isomorphism}, each functional belonging to $H^1(\OO)$ can be realized as an iterated integral $g_j$ of some tuple of $1$-forms. The pairing is given by iterated integration. By Proposition~\ref{prop:base point independence} and Proposition~\ref{prop:Melnikof functional}, both the functionals $g_j$ and $M_\mu$ are base point independent.  

In the chosen basis of $\C H_1(\OO)$ monodromy acts linearly with corresponding matrix $Mon_\alpha$.

Next, we border the Chen pairing matrix, thus obtaining the  \emph{bordered Chen pairing matrix} $C_{\gamma,M}$. Here the last row corresponds to the cycle $\gamma$ and the last column corresponds to taking the first nonzero Melnikov function along the corresponding cycle.

We study the determinant of the bordered Chen matrix $\det(C_{\gamma,M})$.
Note that on one hand side this determinant vanishes identically, because $\gamma$ belongs to the orbit $\OO$.
On the other hand, we develop this determinant with respect to the last column. Denoting $n+1$ the size of the determinant, this gives an expression of the form
\begin{equation}\label{det=0}
a_{n+1}(t) M_\mu (\gamma(t))=
\sum_{j}a_j(t) g_j(\gamma(t)),
\end{equation}
where $g_j$ are iterated integrals of length $\ell\le k$. 	The function $a_{n+1}(t)$ is nonzero on $\C\setminus\Sigma$. Monodromy acts linearly on  \eqref{det=0}.  This is highly non-trivial.  Recall that iterated integrals are in general not linear with respect to product of cycles, instead the shuffling formula \eqref{eq:shuffling formula} applies. Here we recover linearity using Baker-Campbell-Hausdorff formula thanks to the vanishing of both $g_j$ and $M_\mu$  on $K$, see Proposition~\ref{prop:isomorphism}.

Now, by action of monodromy \eqref{det=0} is multiplied by $\det Mon_\alpha$. 
Hence,  the quotients  $b^M_j=\frac{a_j}{a_{n+1}}$ are well-defined,  univalued and of moderate growth, therefore they are rational functions of $t$ with possible poles in $\Sigma$.

\section{Proof of  Theorem~\ref{main}$(ii),(ii')$.}\label{sec: lower bounds proofs}
Assume $k=1$. Then $\gamma\not=0$ in $H_1(F^{-1}(t_0))$. Take $\eta$ such that $M_1=\int_{\gamma}\eta\not=0$, showing that the first nonzero Melnikov function of the deformation $dF+\epsilon\eta=0$ is the abelian integral $M_1$.

Assume that $\kappa>2$. Here we prove that one can find a perturbation $\omega$ such that  $M_2$ is the first non-vanishing Melnikov function and it is not an abelian integral.

By hypothesis, there exists a loop $\delta\in \left(\OO\cap L_2\right)\setminus K$. 
Recall \cite{F,G} that 
$$
M_1(\gamma)=\int_{\gamma}^{}\omega,\,\, M_2(\gamma)=\int_{\gamma}\omega\omega',
$$
where $\omega'$ is the Gelfand-Leray derivative of $\omega$ and  the last formula holds if $M_1(\gamma)\equiv 0$.

\begin{lemma}\label{lem:omega}
	
	There exists $\omega$ such that 
	$$
	M_1(\gamma)=\int_{\gamma}\omega\equiv 0,\qquad \int_{\delta}\omega\omega'\not\equiv 0.
	$$
\end{lemma}

This shows that $M_2$ is the first nonzero Melnikov function (as otherwise it would vanish on $\delta$). Moreover, $M_2$ is not an abelian integral as any abelian integral vanishes on $\delta\in L_2$.

\begin{lemma}\label{lem:Im2}
	\hfill
	\begin{enumerate}
		\item[(i)] $\frac{L_2}{KL_3}$ has no torsion elements.
		\item[(ii)]$\left(\frac{L_2}{KL_3}\right)^*=\left\{\int\eta_i\eta_j\,\vert  \, \forall\gamma\in\OO\, \int_{\gamma}\eta_i=\int_{\gamma}\eta_j=0 \right\}.$
	\end{enumerate} 
\end{lemma}
\begin{proof}
	Choose a basis $\mathcal{B}=\{\gamma_i, \sigma_j\}$ of $H_1(F^{-1}(t_0))$ such that $\{\gamma_i\}$ form a basis of $\operatorname{Im}(\phi_1)$. Then $\{[\gamma_i,\gamma_j],[\gamma_i,\sigma_j],[\sigma_i,\sigma_j]\}$ form a basis of the free abelian group  $L_2/L_3$. The first two types of commutators lie in $K$, so $\frac{L_2}{KL_3}$ is isomorphic  to the subgroup of $L_2/L_3$  generated by $\{[\sigma_i,\sigma_j]\}$, which is hence a free abelian group.

By Chen's theorem, the dual space $\left(L_2/L_3\right)^*$ is spanned by iterated integrals of length $2$. Let $\mathcal{B}^*=\{\omega_i, \eta_j\}$  be a basis of $H^1(F^{-1}(t_0))$ dual to $\mathcal{B}$.
Then
$$
\int_{[\gamma_i,\gamma_j]}\omega_{i'}\omega_{j'}=\pm\delta_{ii'}\delta_{jj'},\,
\int_{[\gamma_i,\gamma_j]}\omega_{i'}\eta_{j'}=\int_{[\gamma_i,\gamma_j]}\eta_{i'}\eta_{j'}=0,
$$
$$
\int_{[\gamma_i,\sigma_j]}\omega_{i'}\omega_{j'}=0,\,
\int_{[\gamma_i,\sigma_j]}\omega_{i'}\eta_{j'}=\pm\delta_{ii'}\delta_{jj'},\,\int_{[\gamma_i,\sigma_j]}\eta_{i'}\eta_{j'}=0,
$$
$$
\int_{[\sigma_i,\sigma_j]}\omega_{i'}\omega_{j'}=
\int_{[\sigma_i,\sigma_j]}\omega_{i'}\eta_{j'}=0,\,\int_{[\sigma_i,\sigma_j]}\eta_{i'}\eta_{j'}=\pm\delta_{ii'}\delta_{jj'}.
$$
so the iterated integrals $\{\int\eta_i\eta_j, i<j\}$ form a basis of $\left(\frac{L_2}{KL_3}\right)^*$.
	\end{proof}

\begin{corollary}
	There exist two forms $\eta_1,\eta_2$, such that $\int_{\gamma}\eta_i=0$, but $\int_{\delta}\eta_1\eta_2\not=0$.
\end{corollary}
\begin{proof}
	As $\delta\not=0$ in $\frac{L_2}{KL_3}$, then the claim follows from Lemma~\ref{lem:Im2},
	\end{proof}
\begin{proof}[Proof of Lemma~\ref{lem:omega}]

Let $\Theta_i\in\Omega^1(\C^2)$, $i=1,2$, be polynomial forms such that their restrictions to $F^{-1}(t_0)$  coincide with $\eta_i$, and, moreover, 
$\int_{\gamma(t)}\Theta_i=0$, for any analytic continuation of $\gamma$. Existence of such forms follows from Corollary~\ref{cor:n1perp}.

Take $\omega=\lambda \Theta_1 +\lambda^{-1}(F-t_0)\Theta_2$, $\lambda>0$. Then 
$$\omega'=\lambda \Theta_1' +\lambda^{-1}(F-t_0)\Theta_2'+\lambda^{-1}\Theta_2,$$
and,  therefore,
$$
\int_{\gamma}\omega\equiv0,\qquad \int_{\delta(t_0)}\omega\omega'=\int_{\delta(t_0)}\eta_1\eta_2+\lambda^2\int_{\delta(t_0)}\eta_1\eta_1'\not=0,
$$
for sufficiently small $\lambda$.
\end{proof}

\section{Examples}
\begin{example}\textit{Hamiltonian triangle }\cite{U1}.
	
Let  $F=y(x^2-(y-3)^2)$, and let $\gamma(t)$ be a continuous family of real periodic orbits surrounding the center bounded by the triangle. In this case, according to  \cite{GI, U1},
	$$
	\C H_1(\OO)=\big\langle\gamma,\delta_1\delta_2\delta_3,[\delta_1,\delta_2]\big\rangle.
	$$
	\begin{figure}[h]
	\begin{center}
		\includegraphics[height=5cm]{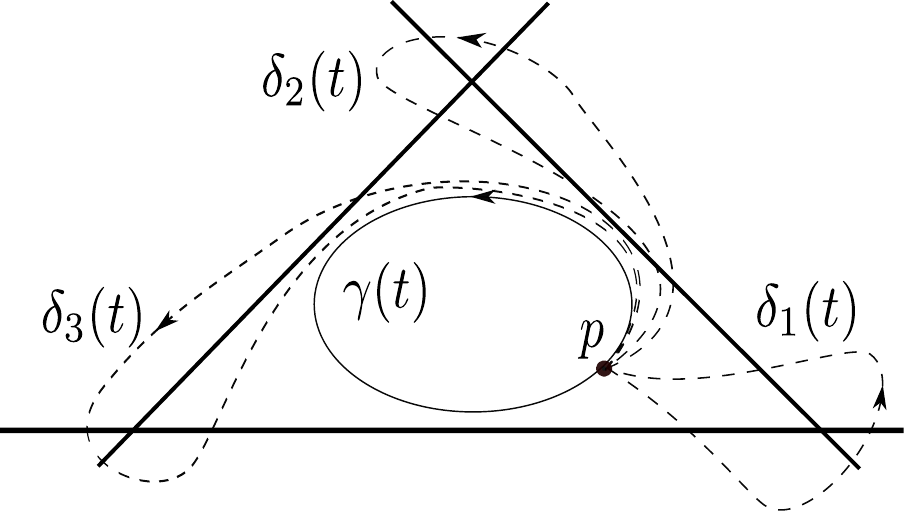}.
	\end{center}
	\caption{The real loop $\gamma(t)$ and the complex vanishing loops $\delta_i(t)$ as elements of $\pi_1\left(F^{-1}(t),p\right)$.}
\end{figure}

	 Hence, $\OO\cap L_3\subset K$ and both $k$ and $\kappa$  are equal to $2$. It follows that  the Melnikov function associated to $\gamma$ and to any polynomial perturbation of $dF=0$, is an iterated integral of length at most $2$. 
	
	On the other hand, in \cite{U1} it is shown that for some perturbation of degree $5$  the first nonzero Melnikov function $M_\mu$ is of length $2$ and is not an abelian integral,  which illustrates  Theorem~\ref{main}$(ii)$.
\end{example}
\begin{example}
	\textit{Product of $(d+1)$ straight lines in general position \cite{U2}.}

	Let $F=\prod_{i=1}^{d+1}f_i$, where $f_i$ are real linear functions in two variables, such that the corresponding lines $\{f_i=0\}$ are in  general position in the real plane. Let $\gamma(t)$ be a continuous family of periodic orbits surrounding a center singular point of the Hamiltonian vector field $F_y\frac{\partial}{\partial x}-F_x\frac{\partial}{\partial y}$.
	
	In \cite{U2},  it is shown that the 1-forms $\{\eta_i=F\frac{df_i}{f_i}\}_{i=1}^{d}$ define a basis of the orthogonal complement $\left(Im\phi_1\right)^\perp$ of the orbit in homology. Moreover,  the iterated integrals $\int \eta_i\eta_j$, with $1\leq i<j\leq d$, are $\C[t,1/t]$-linear independent on the orbit $\OO$. 
	
	Consider the dual space $(\frac{L_2}{L_3K})^*$ of $\frac{L_2}{L_3K}$. It is given by iterated integrals of length $2$ vanishing on $K$. Since an iterated integral $\int \omega_1\omega_2$ is orthogonal to $K$ if and only if the 1-forms $\omega_1$ and $\omega_2$ are orthogonal to the orbit $\OO$, the space $(\frac{L_2}{L_3K})^*$ is generated by the iterated integrals $\int \eta_i\eta_j$,  $1\leq i<j\leq d$. 
	
	If $\operatorname{Im}(\phi_2)=\frac{\OO\cap L_2}{L_3K}\subsetneqq\frac{L_2}{L_3K}$ is a proper linear subspace of $\frac{L_2}{L_3K}$, then there should exist a non-trivial linear combination of the iterated integrals $\int \eta_i\eta_j$ vanishing identically on $\operatorname{Im}(\phi_2)$. This would contradict their linear independence on  $\OO\cap L_2$. Hence  $\OO\cap L_2=L_2$, giving  $k=2$ in Theorem \ref{main}. Therefore, the Melnikov function associated to $\gamma$ and to any polynomial perturbation of $dF=0$, is an iterated integral of length at most $2$. 
\end{example}
\begin{example} $codim(\OO)=1$ \textit{in homology} \cite{PU}.

	Suppose $\pi_1=\langle\gamma_1,...,\gamma_n,\sigma\rangle$ and $\OO=\langle\gamma_1,...,\gamma_n\rangle$.
	Then $K=L_2$, so $L_2\cap \OO\subseteq K$. Therefore, $k$ and $\kappa$, are both equal to $1$ and, by Theorem \ref{main},  the first nonzero Melnikov function $M_\mu$ is an abelian integral.
	
\end{example}

\end{document}